\documentclass[a4paper, 12pt]{article}

\usepackage[sort&compress]{natbib}
\bibpunct{(}{)}{;}{a}{}{,} 

\usepackage{amsthm, amsmath, amssymb, mathrsfs, multirow, url, subfigure}
\usepackage{graphicx} 
\usepackage{ifthen} 
\usepackage{amsfonts}
\usepackage[usenames]{color}
\usepackage{fullpage}
\usepackage{tikz}

\theoremstyle{plain}

\newtheorem{prop}{Proposition}

\theoremstyle{definition}

\theoremstyle{remark}

\newcommand{\prob}{\mathsf{P}} 
\newcommand{\E}{\mathsf{E}}

\newcommand{\bin}{{\sf Bin}}
\newcommand{\unif}{{\sf Unif}}

\newcommand{\bet}{{\sf Beta}}

\newcommand{\TT}{\mathbb{T}}

\renewcommand{\H}{\mathcal{H}}

\newcommand{\prior}{\mathsf{Q}}

\newcommand{\lPi}{\rotatebox[origin=c]{180}{$\mathsf{\Pi}$}} 
\newcommand{\uPi}{\mathsf{\Pi}}

\title{Induction and the rule of succession through a possibilistic inferential model lens}
\author{Ryan Martin\footnote{Department of Statistics, Purdue University, {\tt martinrg@purdue.edu}} \quad Shih-Ni Prim\footnote{Department of Statistical Sciences, Wake Forest University, {\tt prims@wfu.edu}} \quad Max Raner\footnote{Department of Mathematics, Uppsala University, {\tt max.raner@math.uu.se}} \quad Jonathan P. Williams\footnote{Department of Statistics, North Carolina State Univeristy, {\tt jwilli27@ncsu.edu}}}
\date{\today}

\begin{document}

\maketitle 

\begin{abstract}
Induction is the process by which empirical evidence is transformed to knowledge.  Hume famously argued---and Popper and others agree---that there can be no logical justification for induction.  A weaker form of induction, due to Bayes, expresses the aforementioned knowledge in terms of probabilities, and we review some well-known and not-so-well-known criticisms of the Bayesian solution.  We then investigate the relatively new possibilistic inferential model (IM) framework, showing that, in addition to the IM's strong, statistical reliability guarantees that it uniquely enjoys, it is safe from those criticisms that damage the Bayesian foundations.  For illustration, we reconsider the classical sunrise problem and compare our proposed solution with Laplace's famous rule of succession. 

\smallskip

\emph{Keywords and phrases:} Bayesian; Hume; Popper; possibility theory; validity
\end{abstract}

\section{Introduction}
\label{S:intro}

Statistical inference aims to draw conclusions about relevant unknowns based on empirical data.  On what grounds, however, can knowledge about the world or about the future be created based on experience alone?  This is the fundamental question asked by \citet{hume.treatise, hume.enquiry} in {\em A Treatise of Human Nature} and {\em An Enquiry Concerning Human Understanding}, respectively, arguably two of the most significant works of Western philosophy.  As \citet{sep.induction} summarizes,
\begin{quote}
Hume asks on what grounds we come to our beliefs about the unobserved on the basis of inductive inferences. He presents an argument in the form of a dilemma which appears to rule out the possibility of any reasoning from the premises to the conclusion of an inductive inference... 
Therefore, for Hume, the problem remains of how to explain why we form any conclusions that go beyond the past instances of which we have had experience. 
Hume stresses that he is not disputing that we do draw such inferences. The challenge, as he sees it, is to understand the ``foundation'' of the inference---the ``logic'' or ``process of argument'' that it is based upon. 
The problem of meeting this challenge, while evading Hume's argument against the possibility of doing so, has become known as ``the problem of induction.''
\end{quote}
Learning from experience is a cornerstone of both scientific progress and our everyday lives, so the problem of induction is critical.  Indeed, according to Bertrand Russell, if Hume's problem of induction has no solution, then ``there is no intellectual difference between sanity and insanity'' \citep[][p.~699]{russell.history}.  Beyond philosophy, these considerations have practical consequences since, according to Fisher \citep[quoted in][p.~448]{fisher.bio}, statisticians' day-to-day business is inductive inference:
\begin{quote}
For everyone who does habitually attempt the difficult task of making sense of figures is, in fact, assaying a logical process of the kind we call induction, in that he is attempting to draw inferences from the particular to the general, from the sample to the population. 
\end{quote}
Whether it is Hume's purely logical notions of induction or Fisher's practical and statistical notions, an important point is that, in a vacuum, evidence is meaningless.  So, we must always have a meaning-establishing context in which we are evaluating evidence.  In the practical cases that we have in mind here, that context is provided by some sort of model describing the observable phenomena.  Then, when we speak of ``induction,'' we have in mind learning about the unknown parameters that fully determine the posited model.  This is the step that our investigation is focused on.  Surely, however, that model is not a ``correct'' description of reality, so the corresponding inferences drawn are only meaningful within the context of the model.  There is, consequently, a second step in which the model is compared against reality and, while many of the same considerations go into making this assessment, we will not discuss this second step here.

The most widely known proposed solution to the problem of induction is that of \citet{bayes1763}, wherein the probabilities of outcomes given theories---as determined by a context-establishing model as discussed above---are transformed into probabilities of theories given an observed outcome.  In fact, some have argued that Bayes's essay was a direct response to Hume's challenge to solve the problem of induction; for example, \citet[Sec.~4]{zabell1989.rule} traces evidence that suggests Bayes already had his solution by around 1749, only a year after publication of Hume's {\em Enquiry}.  Beyond Bayes's contribution, major theoretical and conceptual advancements were later made by \citet{laplace, laplace.essai}, \citet{definetti1937, definetti.vol1, definetti.vol2}, and others.  Carrying out the Bayesian program when prior information is lacking has been called ``the most important unresolved problem in statistical inference'' \citep{efron.cd.discuss} and, while the solutions championed by \citet{jeffreys1946} and by Berger \citep[e.g.,][]{berger.objective.book} are now quite common in applications, there are still a number of challenges concerning the interpretation of the posterior probabilities and their inductive inference content.  We review these issues in Section~\ref{S:induction}.  

The Bayesian solution to the problem of induction is not the only solution, however.  Another famous solution was advanced by \citet{popper1959}, one that flipped Hume's problem on its head and laid the foundation for what is now known as {\em critical rationalism}.  That is, Popper agrees with Hume that there is no logical justification for induction, but that the growth of scientific knowledge proceeds by creatively formulating bold theories and subjecting those theories to severe tests.  This is based on the principle that a single counterexample is sufficient to disprove a conjecture.  More specifically, according to critical rationalists like Popper as well as \citet{deutsch.fabric, deutsch.infinity}, 
\begin{quote}
... The fate of a theory, its acceptance or rejection, is decided by observation and experiments---by the result of tests.  So long as a theory stands up to the severest of tests we can design, it is accepted; if it does not, it is rejected. But it is never inferred, in any sense, from the empirical evidence.  There is neither psychological nor logical induction.  {\em Only the falsity of the theory can be inferred from empirical evidence, and this inference is a purely deductive one} \citep[][p.~72]{popper1962}.
\end{quote} 
Statisticians will recognize connections to Neyman's frequentism \citep[e.g.,][]{neymanpearson1933, neyman1955} and the perspective of \citet{mayo.book.1996, mayo.book.2018}.  But there is a very wide gap between the Bayesians' fully conditional probabilistic uncertainty quantification and the Neymanians' unconditional procedures that can only suggest if a hypothesis should be rejected or (tentatively) accepted---no uncertainty quantification whatsoever.  We agree with \citet[][p.~122]{hacking.logic.book} when he says that ``Statisticians want numerical measures of the degree to which data support hypotheses,'' and this explains the many attempts, most famously by Fisher, to establish some sort of middle-ground between the Bayesian and frequentist extremes, a ``{\em via media}'' as \citet[][p.~269]{basu1975} puts it.  

The relatively new {\em inferential model} (IM) solution \citep[e.g.,][]{imbasics, imbook} differs markedly from previous attempts in that it quantifies uncertainty but not in terms of ordinary or precise probabilities.  Instead, IMs quantify uncertainty using imprecise probabilities in the general spirit of \citet{levi1980}, \citet{walley1991}, and \citet{imprecise.prob.book}.  But the theory of imprecise-probabilistic inference in the above references is mostly a generalization of Bayesian inference, not so much a middle-ground between frequentists and Bayesians.  In contrast, IMs' output is calibrated in such a way that, among other things, procedures derived from this output are provably reliable in a frequentist sense.  This requires a fundamentally different mathematical framework, and we will focus in Section~\ref{S:im} below on the possibilistic IM formulation recently reviewed in \citet{imreview}.  As we show in Section~\ref{S:im}, in addition to IMs' desirable statistical properties---i.e., performance as a function of data---there are desirable structural, conceptual, and philosophical properties as well.  These properties concern the behavior of IMs' uncertainty quantification as a function of hypotheses about the unknowns when data is treated as fixed.  This investigation reveals that the criticisms surveyed in Section~\ref{S:induction} against Bayesian uncertainty quantification do not apply to the IMs' possibilistic uncertainty quantification.  Moreover, we show that the {\em maxitivity} property of possibility measures offers the user protection from mistaking compatibility of hypotheses with observed evidence for empirical support of those hypotheses.  The point is that one can---intentionally or unintentionally---construct a disjunction of hypotheses that are consistent with the observed evidence and get high Bayesian posterior probability and, therefore, high confidence in the truthfulness of that hypotheses thanks to additivity alone.  This is the take-away message of the {\em false confidence theorem} of \citet{balch.martin.ferson.2017}.  But the maxitivity of the possibilistic IM's uncertainty quantification offers protection from this: only ``genuine support'' (in the sense defined in Proposition~\ref{prop:monotone}), which is much stronger than hypothesis--evidence compatibility, has the power even to make the user's degree of belief in the hypothesis positive given the evidence.   

An important benchmark test case of any framework of data-driven inference is the famous problem introduced by Hume and taken up specifically by (Bayes and) Laplace; namely, that of observing all successes in a sequence of independent Bernoulli trials.  The now-classical solution advanced by Laplace is what we now would describe as the Bayesian posterior mean for the success probability with respect to a uniform prior on the unknown success probability.  That formula  was later referred to as the {\em rule of succession} by \citet[][Ch.~8]{venn.chance} and discussed---sometimes critically---by various authors, including \citet{fisher1973}.  We apply the possibilistic IM machinery in Section~\ref{S:succession}, with a recent decision-theoretic twist, to develop a new solution to this problem, which we argue is superior to maximum likelihood and the Bayesian solutions of Laplace and of Jeffreys.

\section{The problem of induction}
\label{S:induction}

As stated in Section~\ref{S:intro}, the problem of induction concerns offering a logical justification for drawing conclusions about the world and/or the future based on experience alone.  Speaking in the context of inferring cause from observed effects, which is intimately tied to the problem of induction, \citet[][Sec.~IV, Part~II]{hume.enquiry} explains the problem and challenges those with a solution to produce it: 
\begin{quote}
The bread, which I formerly eat, nourished me; this is, a body of such sensible qualities, was, at that time, endued with such secret powers: But does it follow, that other bread must also nourish me at another time, and that like sensible qualities must always be attended with like secret powers? The consequence seems nowise necessary.  At least, it must be acknowledged, that there is here a consequence drawn by the mind; that there is a certain step taken; a process of thought, and an inference, which wants to be explained.  These two propositions are far from being the same, {\em I have found that such an object has always been attended with such an effect}, and {\em I foresee, that other objects, which are, in appearance, similar, will be attended with similar effects}. I shall allow, if you please, that the one proposition may justly be inferred from the other: I know in fact, that it always is inferred.  But if you insist, that the inference is made by a chain of reasoning, I desire you to produce that reasoning.  The connexion between these propositions is not intuitive.  There is a required medium, which may enable the mind to draw such an inference, if indeed it be drawn by reasoning and argument.  What that medium is, I must confess, passes my comprehension; and it is incumbent on those to produce it, who assert, that it really exists, and is the origin of all our conclusions concerning matter of fact.
\end{quote}
In the above passage, Hume is seeking a logical argument by which {\em certain} inferences about the world can be made based on finite experiences.  It is worth mentioning, however, that Hume is open to probabilistic solutions, where uncertainty about the inferences drawn is quantified via probabilities, though he expresses similar doubts about whether a satisfactory solution on these lines can be offered.  Of the many who have taken up Hume's challenge to solve the problem of induction, Bayes's solution via conditional probability is the most famous and most familiar to statisticians.  This is the kind of solution to the problem of induction that we will focus on in the present paper.  To express the relevant nuance and to cover the broadest of scopes, we formulate the details here using propositional logic/calculus, as in \citet[][Sec.~1.2]{jeffreys1961}.

A {\em proposition} $q$ is a statement that is either true or false.  In our present case, where there is underlying uncertainty about the state of the world, it makes sense to think of a proposition as a Boolean function defined over the possible states of the world, so that the truth or falsity of $q$ depends on what state the world is in.  Of course, to a given proposition $q$ there is its {\em negation} $\neg q$; and, for two propositions $q_1$ and $q_2$, there is their {\em conjunction} $q_1 \land q_2$, which is true if and only if both $q_1$ and $q_2$ are true, and their {\em disjunction} $q_1 \lor q_2$, which is true if and only if at least one of $q_1$ and $q_2$ are true.  The mathematics do not require this, but, for our context, it helps to consider three different classes of propositions: hypotheses $h$ that we wish to assess, evidence $x$ (sometimes $e$) that describes observables, and background $b$ consisting of information available {\em a priori}.  Let $\prob$ denote a probability function defined over the set of propositions; to avoid measure-theoretic distractions, we assume that the space of propositions is countable, so it is just a matter of assigning mass to each proposition.  Then the Bayesian solution to the problem of induction is to use the conditional probability 
\[ h \mapsto \prob(h \mid x \land b) := \frac{\prob(h \land x \mid b)}{\prob(x \mid b)} \]
to draw conclusions about hypothesis $h$ based on given evidence $x$ and background $b$.  The idea is that evidence supporting---or at least consistent with---a hypothesis $h$ will amplify, boost, or drive up the probability $\prob(h \mid x \land b)$, thereby offering justification for inferring $h$ based on observation $x$.  

A common \citep[e.g.,][p.~64]{earman.bust} and seemingly compelling illustration of this amplification-based argument for probabilistic induction is in the case where a hypothesis $h$ entails $x$, i.e., $\prob(x \mid h \land b) = 1$; this is related to the rule of succession mentioned in Section~\ref{S:intro} and analyzed in detail in Section~\ref{S:succession} below.  In this case, 
\begin{equation}
\label{eq:amplify}
\prob(h \mid x \land b) = \frac{\prob(x \mid h \land b) \, \prob(h \mid b)}{\prob(x \mid b)} \geq \prob(h \mid b), 
\end{equation}
where the inequality follows because $\prob(x \mid h \land b) = 1$ and $\prob(x \land b) \leq 1$; typically, this latter inequality is strict, in which case the one in the above display is too.  This reveals the aforementioned amplification: the degree of belief in $h$ is greater after observing evidence $x$ entailed by $h$ than it was before.  So, presumably, more evidence consistent with $h$ implies greater degree of belief in $h$, thereby justifying induction.  

There is some trouble here, however, as we can see by probing the calculation in \eqref{eq:amplify} a bit further.  Indeed, if there are two hypotheses $h_1$ and $h_2$ that entail $x$, then 
\begin{equation}
\label{eq:amplify2a}
\frac{\prob(h_1 \mid x \land b)}{\prob(h_2 \mid x \land b)} = \frac{\prob(h_1 \mid b)}{\prob(h_2 \mid b)}. 
\end{equation}
That is, while conditioning on $x$ amplifies the probabilities of both $h_1$ and $h_2$, it does so in a way that preserves the {\em a priori} ranking of $h_1$ and $h_2$, i.e., 
\begin{equation}
\label{eq:amplify2b}
\prob(h_1 \mid x \land b) > \prob(h_2 \mid x \land b) \iff \prob(h_1 \mid b) > \prob(h_2 \mid b). 
\end{equation} 
No matter how strong evidence $x$ is, it does not have the power to change the {\em a priori} rankings assigned to the hypotheses it is entailed by.  \citet[p.~343]{popper1985.realism} says: 
\begin{quote}
[The result in Equation~\eqref{eq:amplify2b}] is shattering. It shows that the favourable evidence [$x$], even though it raises the probability according to [Equation~\eqref{eq:amplify2a}], nevertheless, against first impressions, leaves everything precisely as it was. It can never favour [$h_1$] rather than [$h_2$]. On the contrary, the order which we attached to our hypotheses before the evidence remains. It is unshakable by any favourable evidence. The evidence cannot influence it. 
\end{quote}
Popper goes on to explain why this result is ``shattering'' with two examples.  The first, with some obvious modifications of our own, is the familiar toy example where 
\begin{align*}
h_1 & = \text{all swans are white} \\
h_2 & = \text{all swans are black---except in Raleigh, where they are all white} \\
b & = \text{experiments are conducted exclusively in Raleigh}.
\end{align*}
Since $h_2$ is more flexible than $h_1$, one might be inclined to assign higher prior probability to $h_2$.  In that case, according to \eqref{eq:amplify2b}, no matter how many white swans are observed, the Bayesian can never favor $h_1$ over $h_2$, or even close the {\em a priori} gap between the two.  Popper's second example aims to emphasize that this is not just an issue in toy problems like the first example.  There he takes $h_1$ to be Newton's theory of gravitation, $h_2$ Einstein's, and $x$ the evidence available as of 1917, which was equally compatible with both theories; note that \eqref{eq:amplify2b} does not require that the evidence be entailed by both theories, only that it is equally compatible with both.  So, any Bayesian-derived preference between the two theories would be completely determined by prior beliefs, independent of the available evidence.  Popper further argues that such a situation is not special: for any $h_1$ and $x$, there exists a corresponding $h_2$ such that the evidence $x$ does not inform the preference between $h_1$ and $h_2$. 

Beyond the aforementioned issues with the probabilistic solution to the problem of induction, there are other more subtle concerns like that  illustrated in \citet{popper.miller.1983, popper.miller.1987}.  They start by factoring a hypothesis $h$ with respect to evidence $x$ as 
\begin{equation}
\label{eq:factor}
h = (h \lor \neg x) \land (h \lor x). 
\end{equation}
On the one hand, the second term in this factorization, namely $h \lor x$, is precisely what can be deduced concerning $h$ from observing $x$; that is, if $x$, then $h \lor x$.  On the other hand, the first term, $h \lor \neg x$, is the part of $h$ that cannot be deduced from $x$, which is precisely what induction aims to help assess.  Moreover, according to \citet[][p.~687]{popper.miller.1983}, ``each of the two factors is the weakest proposition strong enough in the presence of the other factor to entail the proposition $h$,'' so this is an efficient/expressive factorization of hypothesis $h$.  The key point is that, if probabilistic induction were possible, then any amplification or lift-off like that observed in \eqref{eq:amplify} should manifest itself as an increase in the probability of the to-be-induced part $h \lor \neg x$, from prior to posterior.  On the contrary, Popper and Miller showed that, for any $h$ and $x$, there is no such amplification and, moreover, there is instead a compression:
\begin{equation}
\label{eq:compress}
\prob(h \lor \neg x \mid x \land b) < \prob(h \lor \neg x \mid b). 
\end{equation}
Popper and Miller also show that $h \lor \neg x$ and $h \lor x$ are probabilistically conditionally independent, given $x$ (and $b$), which implies that 
\[ \prob(h \mid x \land b) = \prob(h \lor \neg x \mid x \land b) \, \prob(h \lor x \mid x \land b). \]
Therefore, in light of \eqref{eq:compress}, any amplification or lift-off on the left-hand side cannot be attributed to induction, or be interpreted as an inductive effect, since it is solely a consequence of an even greater amplification of the second factor on the right-hand side, which is purely deductive.  In summary, \citet[p.~688]{popper.miller.1983} write 
\begin{quote}
[The result in \eqref{eq:compress}] is completely devastating to the inductive interpretation of the calculus of probability. All probabilistic support is purely deductive: that part of a hypothesis that is not deductively entailed by the evidence is always strongly countersupported by the evidence---the more strongly the more the evidence asserts. This is completely general; it holds for every hypothesis
$h$; and it holds for every evidence [$x$], whether it supports $h$, is independent of $h$, or countersupports $h$. 
\end{quote}

Naturally, not everyone agrees with the Popper--Miller analysis and the corresponding ``shattering'' and ``completely devastating'' conclusions \citep[e.g.,][]{jeffrey1984, levi1984, good1984, redhead1985}.  Detailing these counterarguments is beyond the scope of the present paper but, suffice it to say that there is good reason to doubt that the probability calculus solves the problem of induction.  Technical details aside, there is another, far less controversial criticism of the probabilistic solution to the problem of induction.  That is, what makes a theory {\em good} is that it has ``high content,'' that it offers very specific explanations and/or predictions which are easy to test severely.  But theories that make very specific claims are often relatively simple, hence are unlikely to be true in any meaningful sense. For example, the probability assigned to the hypothesis ``the treatment and control groups have identical expected outcomes'' is often zero, independent of the evidence, solely by virtue of corresponding to a measure-theoretically small subset of unknowns. Conversely, the theories that are likely to be true (e.g., treatment and control groups have different expected outcomes) and hence might be assigned high probability are ``low content'' in the sense that they say very little, making vague or trivial predictions that cannot be tested severely.  Since the good theories sought must have low probability of being true, the critical rationalist must not seek theories with high probability, hence the probability assigned to hypotheses is irrelevant: ``Since we aim in science at high content, we do not aim at high probability'' \citep[][p.~386]{popper1962}. 

Popper famously advanced his own solution to the problem of induction, which is obviously not based on probabilistic reasoning.  The next section introduces a form of {\em possibilistic reasoning} which has close connection to Fisher's brand of inductive reasoning and, as we argue below, is both free of the issues with probabilistic induction highlighted above and is consistent with Popper's solution to the problem of induction.  

To summarize: there are some fundamental shortcomings with probabilism as it concerns induction, thereby motivating our investigation into alternative, non-probabilistic modes of uncertainty quantification.  It is also true that Bayesian methods give reasonable answers, at least to certain questions, in most applications.  These two points are not contradictory, since many different methods can ``work well'' while none of them are ``the foundationally correct solution.'' Our efforts here are motivated by the goal of achieving the latter which, of course, must also achieve the former.

\section{Possibilistic inferential models}
\label{S:im}

\subsection{Construction and interpretation}

This section describes the relatively new possibilistic IM framework as recently reviewed in \citet{imreview}.  Here we adopt more-or-less the same perspective taken and notation used in Section~\ref{S:induction}, so this looks a bit different than other presentations; again, this difference is just to highlight the broad scope in which this framework can be applied.  

The key difference between the perspective here and that in Section~\ref{S:induction} above is that here we do not assume existence of prior probabilities associated with the relevant hypotheses to be inferred.  More specifically, we take the perspective common in the statistical literature that prior information about the relevant hypotheses is {\em vacuous}.  Therefore, all that can be said {\em a priori} is that $\prob(h \mid b) \in [0,1]$ for all $h$, with the exception of those trivial hypotheses corresponding to absolute truth and falsity, for which the probability is necessarily 1 and 0, respectively.  More formally, prior knowledge would be described by a lower and upper probability pair $(\lPi, \uPi)$---in the sense of \citet{walley1991}, \citet{levi1980}, and others---with the property that, for all non-trivial hypotheses, $\lPi(h) = 0$ and $\uPi(h) = 1$. A challenge to Bayesian reasoning is that vacuous prior knowledge is too sticky in the sense that every hypothesis with {\em a priori} lower probability 0 has, under (generalized) Bayes updating, {\em a posteriori} lower probability 0 too, no matter what evidence is observed \citep[e.g.,][]{kyburg1987, gong.meng.update}.  But that Bayesian updating fails does not imply that other updating procedures fail too.  What we describe below is a novel, non-Bayesian updating of vacuous prior knowledge in light of observed evidence.  

Without prior probabilities for hypotheses, all that we can model is the evidence given hypotheses, i.e., $x \mapsto \prob(x \mid h \land b)$.  To simplify the notation, we write $\prob_h(x)$ for $\prob(x \mid h \land b)$, dropping the dependence on the background information $b$ and expressing dependence on $h$ as a subscript; this more closely resembles the ``parametric model'' notation familiar to statisticians.  Again, just to keep the presentation simple, suppose that the relevant spaces are finite/countable, so that $\prob_h(x)$ is a mass function for each $h$, i.e., $\sum_x \prob_h(x) = 1$.  Note, also, that the only values of $x$ that could possibly be observed, given the posited model, are those for which $\prob_h(x) > 0$ for some hypothesis $h$. 

Let $\H$ denote the collection of all relevant hypotheses, assumed still to be closed under negation and conjunction.  An important point is that there need not be a model $\prob_h$ for every $h \in \H$.  Instead, there is typically a set of ``elementary hypotheses'' $\H_\text{elem} \subset \H$ for which $\prob_h$ is defined, but inference might aim to go beyond $\H_\text{elem}$.  For example, in the familiar statistical setting, $\H_\text{elem}$ corresponds to simple hypotheses whereas $\H$ itself contains simple and composite hypotheses. It is important to emphasize that the elementary hypotheses are not trivial; in fact, these are the most important in the sense that they   offer direct explanation or prediction of outcomes, and they are the building blocks of all the relevant hypotheses.  Throughout, we reserve the symbol $h$ for general hypotheses in $\H$ and write $\eta$ for elementary hypotheses in $\H_\text{elem}$. 

For any $\eta \in \H_\text{elem}$, and evidence $x$, define the relative likelihood as 
\[ R(x,\eta) = \frac{\prob_\eta(x)}{\max_{\eta' \in \H_\text{elem}} \prob_{\eta'}(x)}. \]
Clearly, if $\hat \eta_x \in \arg\max_\eta \prob_\eta(x)$, then $R(x, \hat \eta_x) = 1$.  Next, define 
\begin{align}
\pi_x(\eta) & = \prob_\eta(\{e: R(e,\eta) \leq R(x,\eta)\}) \notag \\
& = \sum_{e} 1\{R(e,\eta) \leq R(x,\eta)\} \, \prob_\eta(e), \quad \eta \in \H_\text{elem}, \label{eq:contour}
\end{align}
to be the $x$-dependent {\em possibility contour} on $\H_\text{elem}$.  What makes this a genuine possibility contour in the sense of, e.g., \citet{dubois.prade.book}, is that $\pi_x(\hat \eta_x) = 1$.  This property is crucial as it relates to the imprecise-probabilistic interpretation of what follows and, in particular, its coherence properties \citep[e.g.,][]{walley1991, lower.previsions.book, miranda2008}, though we will not get into these details here; see \citet[][App.~D]{imreview}.  While other ranking functions besides the relative likelihood can be used, the choice of ranking function is not arbitrary---one needs to choose the ranking function in such a way to ensure that the normalization condition, $\max_\eta \pi_x(\eta) = 1$, is satisfied; otherwise, the corresponding ``possibilistic inference'' is incoherent in the above sense. 

Statisticians will recognize $\pi_x(\eta)$ as a p-value associated with a significance test at $\eta \in \H_\text{elem}$ with test statistic being the relative likelihood.  This forges a clear connection between the possibilistic IM solution and Fisher's brand of inductive inference as described in, for example, \citet{fisher1973}.  But this approach also has some interesting connections to both Hume and Popper.  First, while Hume has no solution to the problem of induction, no logical justification for why we can transfer past experience to the future, he admits that this can be and is frequently done, and his argument \citep[][Sec.~6]{hume.enquiry} is that the plausibility of $h$ in light of observation $x$ is directly related to the number of past instances where $x \land \eta$ occurs. That is, an accumulation of corroborating cases, or past experiences where $\eta \land x$ held, lead the mind to conclude that $\eta$ is (perhaps highly) plausible in light of a new case with evidence $x$.  Similarly, Popper defines a corroborating case to be one in which the hypothesis $\eta$ is exposed to a severe test and the evidence $x$ fails to refute $\eta$.  It is through an accumulation of many such corroborating cases that $\eta$ is likewise plausible---or tentatively accepted---in the sense that it has ``proved its mettle'' \citep[][p.~10]{popper1959}.  While both Hume and Popper have in mind that only {\em real}---rather than artificial or simulated---experiments or experiences contribute to the growth of knowledge, if the background information $b$ is genuine and if its translation into a probability model $\{\prob_\eta: \eta \in \H_\text{elem}\}$ (dependence on $b$ is suppressed in the notation) is acceptable, which are admittedly ``big ifs,'' then the weighted sum over different values of the evidence in \eqref{eq:contour} is a good proxy for those real experiences.  So, if we define a case $e$, where $R(e,\eta) \leq R(x,\eta)$, to be one that ``corroborates'' $h$ in light of observed $x$, then $\pi_x(\eta)$ in \eqref{eq:contour} is just the accumulation of corroborating cases, akin to what Hume and Popper offer as measures of the plausibility of $\eta$ given $x$.  

To get beyond the elementary hypotheses, it is essential that we can extend the function $\pi_x$ on $\H_\text{elem}$ to $\H$.  For this, we use the {\em extension principle} of possibility theory \citep[e.g.,][]{zadeh1975a}.  That is, we first define a new function $\uPi_x: \H \to [0,1]$ by 
\begin{equation}
\label{eq:upper}
\uPi_x(h) = \max_{\eta \in \H_\text{elem}: \eta \models h} \pi_x(\eta), \quad h \in \H, 
\end{equation}
where ``$\models$'' means ``entails,'' i.e., $\eta \models h$ if and only if $h \lor \neg\eta \equiv \top$, with $\top$ the tautologically true proposition.  The Choquet integral, studied in \citet{lower.previsions.book} and applied in \citet{imdec.ext}, provides extensions beyond upper probabilities to a broader collection of upper expectations; see Section~\ref{S:succession} for an example.

\subsection{Statistical properties}
\label{SS:stat.properties}

Central to the IM framework is a notion of reliability or {\em validity}.  In words, validity means that it is a rare event, with respect to the sampling distribution of evidence, that the IM assigns low possibility to a hypothesis that is true.  Mathematically, 
\begin{equation}
\label{eq:svalid}
\sup_{\eta \in \H_\text{elem}} \prob_\eta(\{x: \pi_x(\eta) \leq \alpha\})
\leq \alpha, \quad \alpha \in [0,1]. 
\end{equation}
That is, observing evidence $x$ that leads to a ``small'' possibility assignment to the elementary $h$ that actually caused it is an event with ``small'' probability.  That the two notions of ``small'' are tied together by a common threshold $\alpha$ is the calibration property necessary to give the possibility assignments a meaningful interpretation.  The result in \eqref{eq:svalid} is an immediate consequence of Fisher's probability integral transform \citep[e.g.,][Theorem~2.1.10]{casella.berger.book} applied to the sampling distribution of the relative likelihood as a function of evidence for a given hypothesis.  

The maxitivity of the possibility measure allows for the following immediate extension of the calibration property in \eqref{eq:svalid} to more general hypotheses:
\begin{equation}
\label{eq:valid}
\sup_{\eta \in \H_\text{elem}} \prob_\eta(\{x: \uPi_x(h) \leq \alpha \text{ for some $h$ with $\eta \models h$}\}) \leq \alpha, \quad
\alpha \in [0,1]. 
\end{equation}
Again, it is a rare event that the possibility assigned to {\em any} true hypothesis is small, where a hypothesis $h$ is ``true'' if and only if the elementary $\eta$ which causes the evidence entails it.  Note that this result is {\em uniform} in hypotheses, which is critical to our claim that the IM solution offers reliable uncertainty quantification in the broad sense that covers the gamut of hypotheses simultaneously.  Since the uniform control is stronger than pointwise control, of course, the corresponding pointwise result also holds:
\[ \sup_{\eta \in \H_\text{elem}: \eta \models h} \prob_\eta(\{x: \uPi_x(h) \leq \alpha\}) \leq \alpha, \quad \alpha \in [0,1], \quad h \in \H. \]
This latter result has parallels to classical hypothesis testing: if we test a hypothesis $h$ versus $\neg h$ using $x$ based on the rule ``reject $h$ if and only if $\uPi_x(h) \leq \alpha$,'' then the above property implies that the test has Type~I error no more than $\alpha$. 

The essential point behind the above validity properties is that the possibility assignments are calibrated with respect to the sampling distribution of the observed evidence, which means that we can control the probability of erroneous inferences.  More specifically, if we take Popper's refutation-based perspective on learning from experience, then the only conclusions that can be drawn are negative and based on a judgment that the possibility assigned to a hypothesis, given the observed evidence, is too small to warrant its further consideration.  Then \eqref{eq:valid} says that such a judgment is unlikely to be wrong, a guarantee that is necessary in order for one to feel confident in any instantiation of Popper's refutation-based solution to the problem of induction.  

As an aside, the property \eqref{eq:svalid} allows us to construct confidence sets based on the IM output.  Indeed, define 
\begin{equation}
\label{eq:conf_set}
C_\alpha(x) = \{\eta \in \H_\text{elem}: \pi_x(\eta) > \alpha\}, \quad \alpha \in [0,1]. 
\end{equation}
Then this defines a $100(1-\alpha)$\% confidence set in the sense that its coverage probability is at least $1-\alpha$ or, equivalently, its non-coverage probability is no more than $\alpha$:
\[ \sup_{\eta \in \H_\text{elem}} \prob_\eta(\{x: C_\alpha(x) \not\ni \eta\}) \leq \alpha. \]
A visualization of the confidence set in \eqref{eq:conf_set} is provided in Figure~\ref{fig:bin_contour} below. 

These important finite-sample guarantees do not come at the expense of the IM being overly conservative.  Indeed, \citet{imbvm.ext} prove a version of the celebrated Bernstein--von Mises theorem that establishes a form of ``large-sample Gaussianity'' with covariance matching the Cram\`er--Rao lower bound.  From this, one can conclude that possibilistic IMs are asymptotically efficient in the same sense as maximum likelihood estimators and Bayesian posterior distributions are asymptotically efficient; for more on the IM/Bayesian/fiducial connection, see \citet{martin.isipta2023, reimagined}. 

Finally, there is also an important counterfactual element to the evaluation of $\pi_x$ and to our proposed brand of possibilistic inference more generally.  Indeed, the contour $\pi_x$ and the corresponding possibility measure $\uPi_x$ depends not just on the actual evidence $x$ obtained but on the other evidence that could have been observed but was not.  This is important because, in our view---and also apparently in Fisher's, Mayo's, and others' views---one cannot reliably or efficiently determine if evidence is incompatible with a hypothesis just from the observed evidence alone: the other possible values of evidence that were not observed are also relevant.  This is related to the still-controversial likelihood principle \citep[e.g.,][]{birnbaum1962, bergerwolpert1984}, and it means that the famous criticism of \citet[][p.~385]{jeffreys1961}---``a hypothesis by that may be true may be rejected because it has not predicted observable results that have not occurred''---applies to our proposed IM solution.   \citet[][Exhibit~v, p.~168--170]{mayo.book.2018} debunks this criticism, and we refer the interested reader to her arguments.  Beyond Mayo's debunking, satisfying the likelihood principle while maintaining reliability requires sacrificing efficiency so, as \citet{martin.basu} argues, it is generally not desirable to satisfy the likelihood principle.  And when it may be desirable, e.g., when the experiment is not sufficiently well spelled out to know what other evidence might have been observed, extensions of the possibilistic IM solution are available and discussed in \citet[][Sec.~6]{martin.basu}.

\subsection{Structural properties}

The discussion in Section~\ref{S:induction} 
highlighted some criticisms against the Bayesian/probabilistic solution to the problem of induction.  Here we establish certain properties of the possibilistic IM solution that, in our interpretation, protect it against the aforementioned criticisms that negatively impact the Bayesian solution.  The IM solution's response to the specific Popper--Miller criticism will be addressed in Section~\ref{SS:pm} below.  

Then the first result shows that, in the basic form of the ``problem of induction,'' where a hypothesis entails the evidence, there is no amplification of the IM's upper probability, which is desirable.  That is, if we believed, {\em a priori}, that $h$ was fully possible, and if it entails the observed $x$, then there's no reason to doubt it and, therefore, it should remain fully possible.  Conversely, if $h$ entails $x$ or even if $h$ and $x$ are not contradictory, there is no amplification whatsoever unless $h$ is entailed by an elementary hypothesis. 

\begin{prop}
\label{prop:upi.noamp}
{\rm (a)} If $h$ entails $x$, and there exists an elementary hypothesis that entails $h$, then $\uPi_x(h)=1$. {\rm (b)} If $h$ and $x$ are not contradictory, but there is no elementary hypothesis that entails $h$, then $\uPi_x(h) = 0$. 
\end{prop}

\begin{proof}
(a) If there exists $\eta \in \H_\text{elem}$ such that $\eta$ entails $h$, then $\eta$ entails $x$ too.  Then $R(x,\eta)=1$, the maximum possible value, and probability that the relative likelihood is no more than its maximum possible value equals 1.  (b) If there is no such $\eta$, then the calculation of $\uPi_x(h)$ involves a maximum over an empty set, which equals 0. 
\end{proof}

Proposition~\ref{prop:upi.noamp}(b) is important and expected, but the succinct statement may not be transparent, so further details may be helpful. Consider the toy example above where, as part of our experiment, we visit a pond in Raleigh to check the color of swans and find that all the swans there are white.  Consider the hypothesis $h = \text{``all cats are black''}$, which is equivalent to ``creatures that are not black are not cats.''  Observing all white swans is consistent with this latter hypothesis in the sense that the two are not contradictory, but, following \citet{hempel1945.1}, one can ask if observing all white swans at the pond contribute anything to our knowledge about the color of cats?  Of course not, and the IM solution agrees with this intuition.  Indeed, the elementary hypotheses offer explanations, or (probabilistic) predictions, of what might be observed in the swan-seeking experiment, corresponding to different proportions of swans that are white versus not white.  A hypothesis about the color of cats cannot be constructed via conjunctions, disjunctions, or negations of elementary hypotheses about the proportions of white swans and, therefore, can have no {\em a priori} support or possibility.  That a hypothesis about cats can be compatible with swan observations should not be enough to change the {\em a priori} status of the former hypothesis, and Proposition~\ref{prop:upi.noamp}(b) establishes this. 

The next results concern the behavior of the IM's corresponding lower probability---or ``necessity measure'' as it is often called in the literature.  This is defined as 
\begin{equation}
\label{eq:lower}
\lPi_x(h) = 1 - \uPi_x(\neg h), \quad h \in \H.
\end{equation}
This so-called {\em conjugacy} property links the lower and upper probabilities together.  Without getting into details about the gambling context, for any $h$, $\uPi_x(h)$ represents the infimum price we would accept from another agent in exchange for making a commitment to pay him $1(h)$ dollars.  In the case of $\neg h$, making a commitment to pay the agent $1(\neg h)=1 - 1(h)$ dollar is equivalent to purchasing a commitment from him that pays us $1(h)$ dollar.  So the supremum price, $\lPi_x(h)$, that we would be willing to pay the other agent for a commitment to pay $1(h)$ dollar must be related to the infimum price we would be willing to accept to make such a commitment ourselves as in \eqref{eq:lower}.  That this supremum selling price is no less than the infimum buying price is clearly necessary in order for the assessment to be coherent, so $\lPi_x(h) \leq \uPi_x(h)$ and, hence, it makes sense to refer to these as lower and upper probabilities. 

A question is: when does amplification of the IM's lower probability occur, or, in other words, when does a degree of $\lPi_x$-lift-off obtains?  Again, that $h$ and $x$ are compatible---not contradictory---should not be enough for amplification.  We can only have genuine support for a hypothesis $h$ if there is no shred of doubt, no better or equally good alternative explanations for the observed evidence.

\begin{prop}
\label{prop:lpi.amp}
$\lPi_x(h) > 0$ if and only if every $\eta \in \H_\text{\rm elem}$ with $\eta \models \neg h$ has $\pi_x(\eta) < 1$.
\end{prop}

\begin{proof}
By conjugacy \eqref{eq:lower}, 
\[ \lPi_x(h) = 1 - \uPi_x(\neg h) = 1 - \max_{\eta \in \H_\text{elem}: \eta \models \neg h} \pi_x(\eta). \]
So, clearly, $\lPi_x(h) > 0$ if and only if the maximum in the above display is less than 1.  This holds if and only if every $\eta \in \H_\text{elem}$ with $\eta \models \neg h$ has $\pi_x(\eta) < 1$.
\end{proof}

The next property is arguably the most important in differentiating possibilistic from probabilistic reasoning, related to the former's ``maxitivity'' compared to the latter's ``additivity.''  Monotonicity of uncertainty quantification is a very natural property: if one hypothesis entails another, then the probability/possibility of the latter should be no smaller than that of the former.  Importantly, a possibility need not be {\em strictly monotone}, i.e., $\uPi_x(h)$ might be equal to $\uPi_x(h')$ even when $h \models h'$, whereas a probability measure often is strictly monotone.  The practical relevance is related to the {\em false confidence theorem} of \citet{balch.martin.ferson.2017}---see, also, \citet{martin.nonadditive, martin.belief2024}---which says that probabilistic uncertainty quantification has a tendency to assign high posterior probabilities to certain hypotheses that are false and, therefore, not directly supported by the evidence.  The key idea is that a false hypothesis can be decomposed into (sometimes infinitely) many elementary components, all of which are false, but are not fully refuted and, therefore, have non-zero posterior probability.  Of course, the sum of many small probabilities can be not-small, leading to potentially erroneous inference.  The false confidence theorem shows that, for some false hypotheses, this potentially misleading behavior is quite common.  For the possibilistic IM, $\lPi_x$ can only increase when $h$ is changed to $h \lor h'$ if $h'$ contains something new that is a strictly better explanation of $x$ than any of the explanations consistent with $h$.  

\begin{prop}
\label{prop:monotone}
Given $h$, $\lPi_x(h \lor h') > \lPi_x(h)$ if and only if the best explanation $\eta \in \H_\text{\rm elem}$ of $x$ that entails $\neg h$ does not entail $\neg h \land \neg h'$.
\end{prop}

\begin{proof}
It is easy to see that 
\begin{align*}
\lPi_x(h \lor h') > \lPi_x(h) & \iff \uPi_x(\neg h \land \neg h') < \uPi_x(\neg h) \\
& \iff \max_{\eta \in \H_\text{elem}: \eta \models \neg h \land \neg h'} \pi_x(\eta) < \max_{\eta \in \H_\text{elem}: \eta \models \neg h} \pi_x(\eta). 
\end{align*}
The latter inequality holds if and only if the maximum on the right is attained at $\eta$ with $\eta \models \neg h$ and that this same $\eta$ has $\eta \not\models \neg h \land \neg h'$ or, in other words, if the best explanation $\eta \in \H_\text{\rm elem}$ of $x$ that entails $\neg h$ does not entail $\neg h \land \neg h'$.
\end{proof}

The point is that one cannot artificially manufacture support by just stitching together sufficiently many hypotheses that are mildly compatible with the evidence.  The hypotheses in question must be compatible with good explanations of $x$.  While the false confidence theorem says it is apparently relatively easy to find sufficiently many hypotheses that are both false and ``compatible'' with $x$ such that their accumulated posterior probability is large, it is much more challenging to find false hypotheses with {\em no good explanation} of $x$ in the true hypothesis that consists of the conjunction of their negations.  This explains why the possibilistic IM is safe from false confidence.

\subsection{Addressing the Popper--Miller criticism}
\label{SS:pm}

Next, we turn our attention to the particular induction-relevant proposition, namely, $h \lor \neg x$, in the factorization \eqref{eq:factor} of $h$ that was at the center of Popper and Miller's analysis.  Recall that, counter-intuitively, the Bayesian's conditional probability contracts on this induction-relevant proposition rather than amplifies, suggesting that Bayesian reasoning fails to solve the problem of induction.  We show below that the IM's assessment is free of the constraint that sets the Bayesian assessment back---that is, the proposition $h \lor \neg x$ can achieve the desired amplification, at least under certain very natural conditions.  

Before justifying this claim, we first need to {\em define} the IM's assessment of a proposition that involves statements about both unknowns and observables.  While the Bayesian formulation constructs a joint probability model for unknowns and observables, so that ``probability of $h \lor \neg x$'' makes sense, conditionally or unconditionally, the IM solution has no such joint distribution.  Indeed, we have only defined $(\lPi_x, \uPi_x)$ for hypotheses $h \in \H$ about unknowns via the extension principle in \eqref{eq:upper}.  So, the evaluation of something like ``$\lPi_x(h \lor \neg x)$'' requires a further extension beyond that in \eqref{eq:upper}.  This highlights a key difference between Bayesian and IM reasoning with respect to the role of conditioning.  

Towards this, let $q$ denote a generic proposition about unknowns and observables, e.g., ``the proportion of white swans observed in a particular pond today in Raleigh exceeds the proportion in all of North America.''  For a given piece evidence $x$ in the support of the model, define the $x$-slice, $q_x$, of $q$ as 
\[ q_x = \bigvee \{\eta \in \H_\text{elem}: \eta \land x \models q\}. \]
This is just the aggregation of those $\eta \in \H_\text{elem}$ for which $\eta \land x$ entails $q$, and hence is a member of $\H$.  Returning to the above example, if $q$ is ``the proportion of white swans observed in a particular pond today in Raleigh exceeds the proportion in all of North America,'' and if $x$ is the stated proportion, then $q_x$ corresponds to the interval $(x, 1]$ consisting of those population proportions that exceed the observed proportion $x$.  Then we define the extension of the upper probability $\uPi_x$ to a joint proposition $q$ as the usual upper probability assigned to the corresponding $x$-slice $q_x$:
\[ \uPi_x(q) := \uPi_x(q_x) = \max_{\eta \in \H_\text{elem}: \eta \land x \models q} \pi_x(\eta). \]
This form of extension aligns with the general notion of {\em focusing}, discussed in \citet[][Sec.~5.1]{martin.partial2} as an alternative to Bayesian-like {\em conditioning}. 

As a special case, consider a generic hypothesis $e$ about observables, potentially different from the observed evidence $x$.  Note, here, that $e$ may be more complex than a single evidence value---think of it as a proposition formed by conjunctions, disjunctions, and negations of various single evidence values.  Then 
\[ (h \lor \neg e)_x = \begin{cases} h & \text{if $x \not\models \neg e$} \\ \top & \text{if $x \models \neg e$}. \end{cases} \]
The factor that determines which of the two cases taken in the above display is a basic compatibility relationship between $x$ and $e$: if $x \models \neg e$, then $x$ and $e$ are incompatible, whereas if $x \not\models \neg x$, then they are compatible---or, not incompatible.  The point is that, when $x$ and $e$ are incompatible in the sense that $x \models \neg e$, then the proposition $h \lor \neg e$ can be deduced from $x$; it is only in the case of compatibility, with $x \not\models \neg e$, where there is an opportunity for induction.  In the former case, where $x \models \neg e$, we get 
\[ \lPi_x(h \lor \neg e) = \lPi_x(\top) = 1, \]
because the proposition is deduced logically.  In the latter case, where $x \not\models \neg e$, we get 
\begin{align*}
\lPi_x(h \lor \neg e) & = \lPi_x\{ (h \lor \neg e)_x \} \\
& = \lPi_x(h) \\
& = 1 - \uPi_x(\neg h) \\
& = 1 - \max_{\eta \in \H_\text{elem}: \eta \models \neg h} \pi_x(\eta). 
\end{align*}
Interestingly, this reveals that the induction-targeting lift-off that was shown by Popper and Miller to fail for Bayesian inference can be realized with IMs.  Given $h$ and $e$, if we are in the case $x \not\models \neg e$ where induction is needed, there is positive $\lPi_x$-lower probability for $h \lor \neg e$ if and only if
\[ \text{every $\eta$ with $\eta \models \neg h$ has $\pi_x(\eta) < 1$}. \]
Of course, if $e = x$, as Popper and Miller specifically consider, then $(h \lor \neg x)_x \equiv h$, since $x$ cannot be incompatible with itself, so deduction is out of reach.  But the case where $(h \lor \neg x)_x \equiv h$ was implicitly already covered in Proposition~\ref{prop:lpi.amp}. 

\begin{prop}
\label{prop:pm}
$\lPi_x(h \lor \neg x) > 0$ if and only if every $\eta \in \H_\text{\rm elem}$ with $\eta \models \neg h$ has $\pi_x(\eta) < 1$.
\end{prop}

For a concrete example, consider the case in Laplace's sunrise example in Section~\ref{S:succession}, where the evidence $x$ consists of $n$ successes in $n$ Bernoulli trials.  As shown in Figure~\ref{fig:bin_contour}, there is an interval $[c,1]$ of $\eta$ values on which the contour is equal to 1; incidentally, $c=\frac{n-1}{n}$, but that is not relevant to the discussion here.  Then the aforementioned lift-off is achieved for any $h$ that corresponds to an interval that contains $[c,1]$.  So, while the Popper--Miller argument reveals that Bayesian solution necessarily experiences the counterintuitive and undesirable contraction phenomenon at propositions of the form $h \lor \neg x$, the same is not true for the IM solution: consistent with one's instinct/intuition, there are hypotheses $h$ for which amplification at $h \lor \neg x$ is achieved.  



\section{The rule of succession}
\label{S:succession}

\subsection{History}

The long-standing sunrise problem poses the question that, if one has seen the sunrise for the first time, how this person would reason the probability of the sun rising again the next day. And, after this person has seen the sunrise for many days and years, what would he conclude the probability of the sun rise yet again the next day? This question was raised in Laplace's 1774 article \citep{stigler.laplace} and his 1814 book \textit{Essai Philosophique sur les Probabilit\'es} \citep{laplace.essai}. This question was originally raised by Hume \citep[e.g.,][p.~173]{zabell.attribution}, and Laplace's answer \citep[][p.11]{laplace.essai} was as follows:
\begin{quote}
If we place the dawn of history at 5,000 years before the present date, we have 1,826,213 days on which the sun has constantly risen in each 24 hour period. We may therefore lay odds of 1,826,214 to 1 that it will rise again tomorrow.
\end{quote}
This triggered a long debate over whether his reasoning was sound, even though he did recognize that ``this number would be incomparably greater for one who, perceiving in the coherence (or totality) of phenomena the principle regulating days and seasons, sees that nothing at the present moment can check the sun's course'' \citep[][p.~11]{laplace.essai}. As Laplace himself states in this passage, and as \citet{zabell.attribution} clearly explains, Bayes, Price, and Buffon all provided slightly different versions of answers before Laplace, and the act of using past events and prior beliefs exemplified the use of induction to reason about the world in the Eighteenth century. 

\subsection{Setup}

Here we formalize the sunrise problem into a statistical setting. Let $X \sim \bin(n, \Theta)$ denote the number of successes in $n$ independent and identically distributed trials, with known number of trials $n$ and unknown success probability $\Theta$. If all $n$ observations are successes---meaning $X=n$---we have exactly the sunrise problem. To connect to the notation in Sections~\ref{S:induction} and \ref{S:im}, the evidence $x$ is the number of successes, the elementary hypotheses $\H_\text{elem}$ corresponds to the parameter space $[0,1]$ of the binomial model, the generic hypotheses $\H$ corresponds to the power set of $[0,1]$, and the background $b$ consists of contextual information stating that the experiment involves a sequence of $n$ many independent Bernoulli trials with a fixed-but-unknown success probability $\Theta$.  This is the meaning- or context-establishing model as described in Section~\ref{S:intro}. With respect to this context, which does not permit changes to the experiment from one trial to the next, an estimate of the success probability $\Theta$ can naturally be used for making predictions about success in the next trial. It is in this sense that estimation of $\Theta$ is related to induction.

When $x=n$, the maximum likelihood estimator is $\hat{\theta}_\text{\sc mle}=x/n=1$. Mathematically, this is uncontroversial, but the philosophical implications are problematic: concluding that the success probability is 1 implies predictions are certain and that induction has been realized, i.e., certain knowledge about the unknown $\Theta$ has been reached based on finite experience.  Since certain induction cannot be realized in finite experience, the method that leads to this conclusion must be flawed. 

Laplace's rule of succession \citep{laplace.essai, zabell1989.rule} can be viewed as an alternative to the aforementioned maximum likelihood solution.  This rule is based on Bayes's formula and an assumption that $\Theta \sim \unif(0,1)$ {\em a priori}.  More generally, if 
\[ (X \mid \Theta=\theta) \sim \bin(n,\theta) \quad \text{and} \quad \Theta \sim \bet(a,b), \]
then $(\Theta \mid X=x) \sim \bet(x+a, n-x+b)$, and, consequently, the Bayes estimator---the minimizer of the posterior risk with respect to squared-error loss---is
\[ \hat\theta_\text{\sc bayes}(x) = \E(\Theta \mid X=x) = \frac{x + a}{n + a + b}. \]
Following Laplace, if we take $a=b=1$, which corresponds to a uniform prior, then, in the case $x=n$, we get the Bayes estimator 
\[ \hat\theta_\text{\sc lap}(x) = \frac{n+1}{n+2}. \]
Notice that this is strictly less than 1 for every finite $n$, consistent with the understanding that no finite experience of all successes leads to a conclusion that success is certain, hence is free of the ideological dissonance associated with the maximum likelihood estimator.  Other, non-uniform choices of prior distribution can be used; for example, the prior recommended by \citet{jeffreys1946} corresponds to $a=b=\frac12$, which yields the estimator 
\[ \hat\theta_\text{\sc jef}(x) = \frac{n+1/2}{n+1}, \]
which is larger than $\hat\theta_\text{\sc lap}$ but also strictly less than 1 for all $n$. 

The maximum likelihood estimator, as the name suggests, is mainly interpreted as an estimator of the unknown $\Theta$.  In our binomial setting, such an estimator inherits an interpretation as a predictive probability for future trials, but this clearly does not properly account for the uncertainty associated with the estimation of $\Theta$.  An advantage of the Bayes estimator is that it can be directly interpreted as a genuine predictive probability, one that incorporates uncertainty about the underlying $\Theta$ through integration with respect to a posterior distribution.  Indeed, if $Y$ denotes the outcome of a future trial, e.g., whether sun rises tomorrow or not, then it can be shown that
\[ \hat\theta_\text{\sc bayes}(x) \equiv \prob_x(Y=1), \]
where $\prob_x$ is the posterior predictive distribution of $Y$, given $X=x$, with respect to the posited Bayes model.  This interpretation of the Bayes estimator as a predictive probability helps to explain why the Laplace and Jeffreys solutions stay strictly less than 1.  Despite these benefits, however, the fact that two different numerical answers can arise from ``noninformative'' priors with the same data sparks doubts---if both priors suggest ignorance about $\Theta$, should they not arrive at the same solution when applied to the same data? If the maximum likelihood, Laplace--Bayes, and Jeffreys estimators are all unsatisfactory, then apparently a new solution is needed.

\subsection{Possibilistic IM solution}

In the binomial model, the probability mass function is 
\[ p_\theta(x) = \binom{n}{x} \theta^x (1-\theta)^{n-x}, \quad x=0,1,\ldots,n, \quad \theta \in [0,1]. \]
When $X=x$ is observed, the maximum likelihood estimator is $\hat\theta_x = x/n$.  Then the relative likelihood is
\[
R(x,\theta) = \frac{p_\theta(x)}{p_{\hat\theta_x}(x)} = \left( \frac{n\theta}{x} \right)^x \left( \frac{n-n\theta}{n-x}\right)^{n-x}.
\]
The contour function, as shown in \eqref{eq:contour}, can be evaluated as
\begin{equation}
\label{eq:bin_contour}
\pi_x(\theta) = \sum_{y=0}^n 1\{ R(y, \theta) \leq R(x, \theta) \} \, p_\theta(y).
\end{equation}
Figure~\ref{fig:bin_contour} shows the contour function in the all-successes case where $x=n=10$. The maximization step is demonstrated by finding the possibility of the hypothesis $H=\{\Theta\in[0.7, 0.9]\}$ (the red segment): $\uPi_x(H) = \sup_{0.7 \le \theta \le 0.9} \pi_x(\theta) \approx 0.42$.  And since $H$ excludes those values of $\theta$ with $\pi_x(\theta)=1$, note that no lift-off is achieved in the sense that $\uPi_x(H)$ and the {\em a priori} lower probability assigned to $H$ agree and both are equal to 0, as predicted by Proposition~\ref{prop:lpi.amp}.  

\begin{figure}[t]
    \centering
    \includegraphics[width=0.57\linewidth]{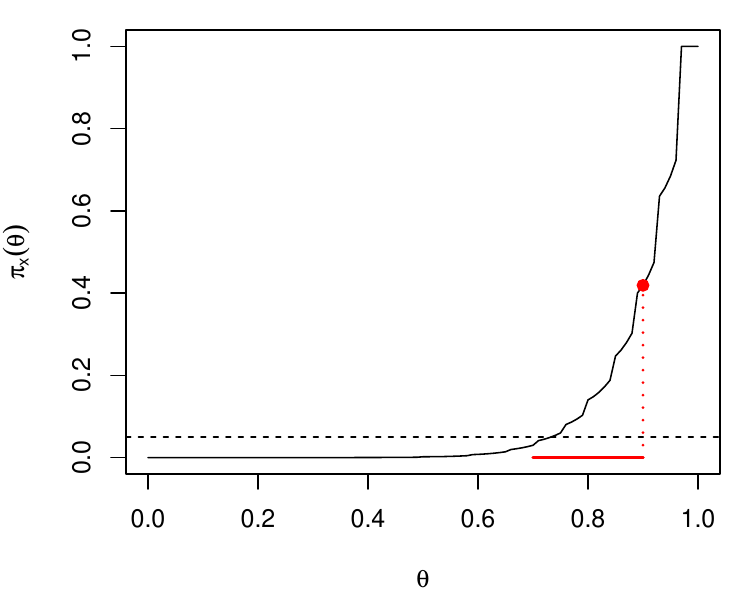}
    \caption{Contour function of a binomial distribution with $n=x=10$. The red line segment represents a hypothesis $h$ = ``$\Theta \in [0.7, 0.9]$,'' and the red point represents the possibility assigned to this hypothesis, $\uPi_x(h) \approx 0.42$. The horizontal dashed line at $\alpha=0.05$ determines the 95\% confidence set $C_{0.05}(n) = [0.733, 1]$ as in \eqref{eq:conf_set}.
    } 
    \label{fig:bin_contour}
\end{figure}

As hinted at in Section~\ref{S:im}, a major advantage of having a full-blown quantification of uncertainty is that one can extend the possibilistic IM beyond assessments of hypotheses, analogous to the extension of probabilities to expected values via Lebesgue integration.  Here, that extension boils down to the supremum of the ordinary expectation among those probability distributions dominated by the IM's possibility measure.  In many cases, including the present one, this complex upper expectation corresponds to a {\em Choquet integral} with a relatively simple form.  According to \citet[][Prop.~15.42]{lower.previsions.book}, for non-negative functions $f: \TT \to [0,\infty)$ (satisfying certain integrability conditions), the upper expectation of $f$, with respect to $\uPi_x$, denoted by $\uPi_x f$, is 
\begin{equation}
\label{eq:choquet}
\uPi_x f = \int_0^1 \sup_{\theta: \pi_x(\theta)>s} f(\theta) \, ds,
\end{equation}
where $\pi_x$ is the contour corresponding to $\uPi_x$. 

Choosing an action based only on a maximizer of the contour, such as the MLE, would discard the remaining uncertainty represented by $\uPi_x$. Following \citet{imdec.ext}, we instead choose an action by minimizing the upper expected loss in \eqref{eq:choquet}, analogous to classical von Neumann--Morgenstern decision theory. For a non-negative loss, $\ell_a$, where $\ell_a(\theta)$ represents the cost of taking action $a$ when the world is in state $\theta \in \TT$, this gives
\begin{align*}
\hat\theta_\text{\sc IM}(x) &= \arg\min_a \uPi_x \ell_a = \arg \min_a \int_0^1 \sup_{\theta: \pi_x(\theta)>s} \ell_a(\theta) \, ds. 
\end{align*}
This typically has no closed-form expression, even in the present binomial application with squared-error loss, but is relatively easy to find numerically. 

A key observation is that many of the appealing interpretations of the Bayesian solution also apply to the IM solution, but without requiring a commitment to a prior distribution.  Skipping over the technical details, it can be shown that $\hat\theta_\text{\sc im}(x)$ is the mean corresponding to a particular ``posterior distribution'' $\prior_x$ contained in the so-called {\em credal set} associated with $\uPi_x$; see \citet{imreview} for details about the IM credal sets.  This $\prior_x$ need not have the {\em likelihood-times-prior} form of a Bayesian posterior for any prior, hence the quotation marks around ``posterior.'' But being a $x$-dependent probability distribution for $\Theta$, it propagates forward to a predictive probability for $Y$, i.e., $\hat\theta_\text{\sc im}(x) = \prior_x(Y=1)$.  So, just like the Bayes estimators, and for exactly the same reasons, the IM estimator $\hat\theta_\text{\sc im}(x)$ is bounded away from unity and, therefore, free of the logical dissonance associated with maximum likelihood estimators.  Again, this superiority is realized without having to commit to any Bayesian prior distribution.

Figure~\ref{fig:laplace_IM} compares the IM estimator with the Laplace--Bayes and Jeffreys estimator in the all-successes ``$x=n$'' case over a range of values of $n$.  The Jeffreys estimator is the largest, the Laplace--Bayes estimator the smallest, and the IM estimator is generally in between the other two; at $n=2$, the Laplace--Bayes estimator and the IM estimator are equal (0.75) but everywhere else the IM estimator is larger.  While the numerical differences are small, this illustration provides a backdrop for a foundational comparison between these different approaches.  First, there seems to be no compelling justification for choosing a prior that shrinks the posterior mean towards $\frac12$ versus one that shrinks towards 0 or 1, as the Laplace--Bayes and Jeffreys solutions do, respectively.  Second, when no prior information is available, the solution should not proceed by manufacturing artificial knowledge for inclusion in the analysis.  Instead, use of a reliable quantification of uncertainty like the proposed IM would, even in the extreme case of all successes, allow for the *possibility* that $\Theta$ is less than 1. Therefore, any proper risk assessment, like that in \eqref{eq:choquet}, would naturally shrink the estimator away from the purely data-driven maximum likelihood estimator at the boundary without the need of any intervention (e.g., inclusion of artificial prior information) from the data analyst.  

\begin{figure}
    \centering
    \scalebox{0.7}{\includegraphics{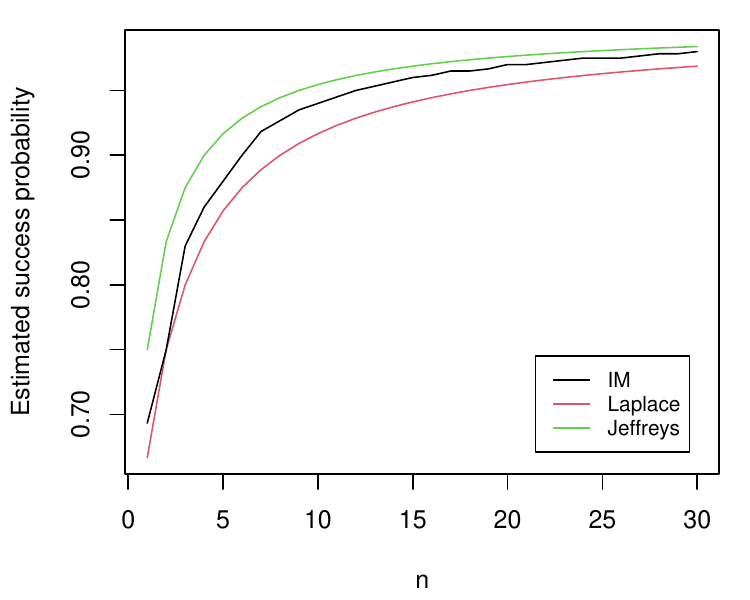}}
    \caption{Comparison of Laplace-Bayes, IM, and Bayes with Jeffreys prior estimators for $n$ ranging from $1$ to $30$.}
    \label{fig:laplace_IM}
\end{figure}

\section{Conclusion}
\label{S:discuss}

The present paper considers Hume's problem of induction through a possibility-theoretic inferential model (IM) lens.  What we found through this investigation is that, in addition to the desirable statistical properties enjoyed by the IM solution, which have been highlighted elsewhere in the literature \citep[e.g.,][]{imreview}, there are some appealing conceptual and philosophical properties as well.  In particular, we showed that the shortcomings of the probabilistic solution to the problem of induction to not apply to the IM solution and highlighted some of the structural benefits of the IM's possibilistic form.  We also proposed a new and principled solution to the sunrise problem of Laplace, one that corresponds to finding an action that minimizes an ``upper expected loss'' determined by the IM's uncertainty quantification, \`a la von Neumann and Morgenstern.  A Bayesian analysis would require some choice of prior distribution and, counterintuitively, even a choice between ``default'' or ``non-informative'' priors affects the Bayes estimator.  Aside from the IM's desirable statistical and philosophical properties that the Bayesian solution cannot match, the proposed estimator is superior, we argue, because the IM solution offers the appeal of a fully conditional Bayesian-like solution but without having to choose a prior---``default'' or otherwise. 

Induction and, more generally, the logic and philosophy of science are age-old considerations, but they are as important and relevant now than ever; see the new book {\em The Laws of Thought} \citep{griffiths.laws}, a modern twist on Boole's classic with the same title.  Indeed, artificial intelligence purports to create knowledge out of ``past experiences''---the data it was trained on.  As Hume and others have taught us, those inferences cannot be certain and, consequently, there is a pressing need for reliable methods for quantifying that uncertainty, to guide the growth of knowledge.  Our proposed possibilistic IM framework is a promising alternative to classical probabilistic reasoning, as we have shown here and elsewhere, but challenges remain.  In particular, the problem setup here in Section~\ref{S:im} assumes vacuous {\em a priori} knowledge.  Such an assumption is more reasonable than the opposite extreme that assumes knowledge can be encoded by a precise prior probability distribution, but is still quite unrealistic: if we believe that knowledge can grow and is growing, then presumably what we know today is different from what we knew yesterday; so even if knowledge was vacuous yesterday, it may no longer be so today.  Again, it cannot be assumed that today's updated knowledge can be described via precise probability, but a possibilistic description is more realistic. What is needed, then, is a reliable strategy for taking potentially non-vacuous {\em a priori} beliefs encoded as possibility and updating them to a new and more informed possibilistic quantification of uncertainty based on new evidence.  Novel developments along these lines have already been made \citep[e.g.,][]{martin.partial2}, but this is a focus of ongoing work.

\section*{Acknowledgments}

 This work is partially supported by the U.S.~National Science Foundation, through the grant DMS--2412628.

\bibliographystyle{apalike}
\bibliography{mybib}

@Article{balch.martin.ferson.2017,
  author  = {Balch, Michael S. and Martin, Ryan and Ferson, Scott},
  title   = {Satellite conjunction analysis and the false confidence theorem},
  journal = {Proc. Royal Soc. A},
  year    = {2019},
  volume  = {475},
  number  = {2227},
  pages   = {2018.0565},
}

@Article{basu1975,
  author   = {Basu, D.},
  journal  = {Sankhy\={a} Ser. A},
  title    = {Statistical information and likelihood},
  year     = {1975},
  issn     = {0581-572X},
  note     = {Discussion and correspondance between Barnard and Basu},
  number   = {1},
  pages    = {1--71},
  volume   = {37},
  fjournal = {Sankhy\={a} (Statistics). The Indian Journal of Statistics. Series A},
  mrclass  = {62A10},
  mrnumber = {440747},
}

@Article{bayes1763,
  author    = {Bayes, T.},
  journal   = {Philos. Trans. Roy. Soc.},
  title     = {An essay towards solving a problem in the doctrine of chances},
  year      = {1764},
  pages     = {370--418},
  volume    = {53},
}

@Book{berger.objective.book,
  author    = {Berger, J. and Bernardo, J. and Sun, D.},
  publisher = {World Scientific Publishing Co.},
  title     = {Objective {B}ayesian {I}nference},
  year      = {2024},
}

@Book{bergerwolpert1984,
  author     = {Berger, James and Wolpert, Robert},
  publisher  = {Institute of Mathematical Statistics},
  title      = {The {L}ikelihood {P}rinciple},
  year       = {1984},
  address    = {Hayward, CA},
  isbn       = {0-940600-06-4},
  series     = {Institute of Mathematical Statistics Lecture Notes---Monograph Series, 6},
  mrclass    = {62A15 (62A20)},
  mrnumber   = {MR773665},
  mrreviewer = {G. K. Robinson},
  pages      = {xi+206},
}

@ARTICLE{birnbaum1962,
  author = {Birnbaum, Allan},
  title = {On the foundations of statistical inference},
  journal = {J. Amer. Statist. Assoc.},
  year = {1962},
  volume = {57},
  pages = {269--326},
  fjournal = {Journal of the American Statistical Association},
  issn = {0162-1459},
  mrclass = {62.05},
  mrnumber = {0138176}
}

@Book{fisher.bio,
  author    = {Box, Joan Fisher},
  publisher = {John Wiley \& Sons, New York-Chichester-Brisbane},
  title     = {R. {A}. {F}isher. {T}he {L}ife of a {S}cientist},
  year      = {1978},
  isbn      = {0-471-09300-9},
  mrclass   = {62-03 (01A70 92-03)},
  mrnumber  = {500579},
  pages     = {xii+512 pp. (1 plate)},
}

@Book{casella.berger.book,
  author    = {Casella, George and Berger, Roger L.},
  publisher = {Wadsworth \& Brooks/Cole Advanced Books \& Software, Pacific Grove, CA},
  title     = {Statistical {I}nference},
  year      = {1990},
  isbn      = {0-534-11958-1},
  series    = {The Wadsworth \& Brooks/Cole Statistics/Probability Series},
  mrclass   = {62-01},
  mrnumber  = {1051420},
  pages     = {xviii+650},
}

@Book{deutsch.fabric,
  author    = {Deutsch, David},
  publisher = {Penguin Books},
  title     = {The {F}abric of {R}eality},
  year      = {1998},
  address   = {London},
  isbn      = {978-0140275414},
}

@Book{deutsch.infinity,
  author    = {Deutsch, David},
  publisher = {Penguin Books, New York},
  title     = {The {B}eginning of {I}nfinity},
  year      = {2011},
  isbn      = {978-0-14-312135-0},
  mrclass   = {81P05 (00A09 00A79)},
  mrnumber  = {2984795},
  pages     = {viii+487},
}

@Book{dubois.prade.book,
  title     = {Possibility {T}heory},
  publisher = {Plenum Press, New York},
  year      = {1988},
  author    = {Dubois, Didier and Prade, Henri},
  mrclass   = {90A05 (04A72 68P20 90-02 90C70 93C42 94D05)},
  mrnumber  = {1104217},
  pages     = {xvi+263},
}

@Book{earman.bust,
  author    = {Earman, John},
  publisher = {MIT Press, Cambridge, MA},
  title     = {Bayes or {B}ust?},
  year      = {1992},
  isbn      = {0-262-05046-3},
  mrclass   = {60A05 (03B48 62A15)},
  mrnumber  = {1170349},
  pages     = {xvi+272},
}

@Article{efron.cd.discuss,
  author   = {Efron, Bradley},
  title    = {Discussion: ``{C}onfidence distribution, the frequentist distribution estimator of a parameter: a review'' [MR3047496]},
  journal  = {Int. Stat. Rev.},
  year     = {2013},
  volume   = {81},
  number   = {1},
  pages    = {41--42},
  fjournal = {International Statistical Review. Revue Internationale de Statistique},
  mrclass  = {62F25 (62-02 62A01 62F10 62F15)},
  mrnumber = {3047498},
}

@Article{definetti1937,
  author   = {de Finetti, Bruno},
  title    = {La pr\'{e}vision : ses lois logiques, ses sources subjectives},
  journal  = {Ann. Inst. H. Poincar\'{e}},
  year     = {1937},
  volume   = {7},
  number   = {1},
  pages    = {1--68},
  fjournal = {Annales de l'Institut Henri Poincar\'{e}},
  issn     = {0365-320X},
  mrclass  = {DML},
  mrnumber = {1508036},
}

@Book{definetti.vol2,
  author    = {de Finetti, Bruno},
  publisher = {John Wiley \& Sons, Ltd., Chichester},
  title     = {Theory of {P}robability. {V}ol.\ 2},
  year      = {1990},
  isbn      = {0-471-92612-4},
  note      = {A critical introductory treatment, Translated from the Italian and with a preface by Antonio Mach{\`{\i}} and Adrian Smith, With a foreword by D. V. Lindley, Reprint of the 1975 translation},
  series    = {Wiley Classics Library},
  mrclass   = {60A05 (01A75)},
  mrnumber  = {1093667},
  pages     = {xviii+375},
}

@Book{definetti.vol1,
  author    = {de Finetti, Bruno},
  publisher = {John Wiley \& Sons, Ltd., Chichester},
  title     = {Theory of {P}robability. {V}ol.\ 1},
  year      = {1990},
  isbn      = {0-471-92611-6},
  series    = {Wiley Classics Library},
  mrclass   = {60A05 (01A75)},
  mrnumber  = {1093666},
  pages     = {xx+300},
}

@BOOK{fisher1973,
  title = {Statistical Methods and Scientific Inference},
  publisher = {Hafner Press},
  year = {1973},
  author = {Fisher, Ronald A.},
  pages = {viii+182},
  address = {New York},
  edition = {3rd},
  mrclass = {62-01},
  mrnumber = {MR0346955}
}

@Article{gong.meng.update,
  author   = {Gong, Ruobin and Meng, Xiao-Li},
  title    = {Judicious judgment meets unsettling updating: dilation, sure loss and {S}impson's paradox},
  journal  = {Statist. Sci.},
  year     = {2021},
  volume   = {36},
  number   = {2},
  pages    = {169--190},
  issn     = {0883-4237},
  fjournal = {Statistical Science. A Review Journal of the Institute of Mathematical Statistics},
  mrclass  = {62C10 (62A01 62C86)},
  mrnumber = {4255191},
}

@Article{good1984,
  author  = {Good, Irving J.},
  journal = {Nature},
  title   = {The impossibility of inductive probability},
  year    = {1984},
  number  = {2},
  pages   = {434},
  volume  = {310},
}

@Book{griffiths.laws,
  author    = {Tom Griffiths},
  publisher = {Henry Holt and Co.},
  title     = {The {L}aws of {T}hought: {T}he {Q}uest for a {M}athematical {T}heory of the {M}ind},
  year      = {2026},
  isbn      = {1250358353},
}

@Book{hacking.logic.book,
  author    = {Hacking, Ian},
  publisher = {Cambridge University Press, Cambridge-New York-Melbourne},
  title     = {Logic of {S}tatistical {I}nference},
  year      = {1976},
  mrclass   = {62AXX},
  mrnumber  = {0391307},
  pages     = {ix+232},
}

@Article{hempel1945.1,
  author    = {Hempel, Carl Gustav},
  journal   = {Mind},
  title     = {Studies in the logic of confirmation ({I}.)},
  year      = {1945},
  number    = {213},
  pages     = {1--26},
  volume    = {54},
  publisher = {Oxford University Press},
}

@InCollection{sep.induction,
  author       = {Henderson, Leah},
  booktitle    = {The {Stanford} Encyclopedia of Philosophy},
  publisher    = {Metaphysics Research Lab, Stanford University},
  title        = {{The Problem of Induction}},
  year         = {2024},
  edition      = {{W}inter 2024},
  editor       = {Edward N. Zalta and Uri Nodelman},
  howpublished = {\url{https://plato.stanford.edu/archives/win2024/entries/induction-problem/}},
}

@Book{hume.treatise,
  author    = {Hume, David},
  publisher = {Penguin Classics},
  title     = {A {T}reatise of {H}uman {Nature}},
  year      = {1739},
  note      = {1985 reprint},
}

@Book{hume.enquiry,
  author    = {Hume, David},
  publisher = {Hackett Publishing Company},
  title     = {An {E}nquiry {C}oncerning {H}uman {U}nderstanding},
  year      = {1748},
  address   = {Indianapolis},
  isbn      = {0915144166},
  note      = {1977 reprint},
}

@Article{jeffrey1984,
  author  = {Jeffrey, Richard C.},
  journal = {Nature},
  title   = {The impossibility of inductive probability},
  year    = {1984},
  number  = {2},
  pages   = {433},
  volume  = {310},
}

@Article{jeffreys1946,
  author   = {Jeffreys, Harold},
  journal  = {Proc. Roy. Soc. London Ser. A},
  title    = {An invariant form for the prior probability in estimation problems},
  year     = {1946},
  issn     = {0962-8444},
  pages    = {453--461},
  volume   = {186},
  fjournal = {Proceedings of the Royal Society. London. Series A. Mathematical, Physical and Engineering Sciences},
  mrclass  = {62.0X},
  mrnumber = {17504},
}

@Book{jeffreys1961,
  author    = {Jeffreys, Harold},
  publisher = {The Clarendon Press, Oxford University Press, New York},
  title     = {Theory of {P}robability},
  year      = {1998},
  isbn      = {0-19-850368-7},
  note      = {Reprint of the 1983 edition},
  series    = {Oxford Classic Texts in the Physical Sciences},
  mrclass   = {62-01 (01A75 03B48 62-02 62A01 62F15)},
  mrnumber  = {1647885},
  pages     = {xii+459},
}

@Article{kyburg1987,
  author   = {Kyburg, Jr., Henry E.},
  title    = {Bayesian and non-{B}ayesian evidential updating},
  journal  = {Artificial Intelligence},
  year     = {1987},
  volume   = {31},
  number   = {3},
  pages    = {271--293},
  issn     = {0004-3702},
  fjournal = {Artificial Intelligence. An International Journal},
  mrclass  = {60A05 (62A15 68T30 90D12)},
  mrnumber = {881287},
}

@Book{laplace,
  author    = {Laplace, Pierre S.},
  publisher = {Courcier},
  title     = {Th\'{e}orie {A}nalytique des {P}robabilit\'{e}s},
  year      = {1812},
  address   = {Paris},
}

@Book{laplace.essai,
  author    = {Laplace, Pierre S.},
  publisher = {Springer--Verlag},
  title     = {Philosophical {E}ssay on {P}robabilities},
  year      = {1995},
  address   = {New York},
  note      = {Translated from the fifth French edition of 1825 by Andrew I. Dale},
}

@Book{levi1980,
  author    = {Levi, I.},
  publisher = {The MIT Press},
  title     = {The {E}nterprise of {K}nowledge},
  year      = {1980},
  address   = {Boston},
}

@Article{levi1984,
  author  = {Levi, Isaac},
  journal = {Nature},
  title   = {The impossibility of inductive probability},
  year    = {1984},
  number  = {2},
  pages   = {433},
  volume  = {310},
}

@Article{martin.nonadditive,
  author   = {Martin, Ryan},
  title    = {False confidence, non-additive beliefs, and valid statistical inference},
  journal  = {Internat. J. Approx. Reason.},
  year     = {2019},
  volume   = {113},
  pages    = {39--73},
  fjournal = {International Journal of Approximate Reasoning},
  mrclass  = {62A01},
  mrnumber = {3979518},
}

@Unpublished{martin.partial2,
  author = {Martin, Ryan},
  note   = {{\tt arXiv:2211.14567}},
  title  = {Valid and efficient imprecise-probabilistic inference with partial priors, {II}. {G}eneral framework},
  year   = {2022},
}

@InProceedings{martin.isipta2023,
  author    = {Martin, Ryan},
  booktitle = {Proceedings of the Thirteenth International Symposium on Imprecise Probability: Theories and Applications},
  title     = {Fiducial inference viewed through a possibility-theoretic inferential model lens},
  year      = {2023},
  editor    = {Miranda, Enrique and Montes, Ignacio and Quaeghebeur, Erik and Vantaggi, Barbara},
  month     = {11--14 Jul},
  pages     = {299--310},
  publisher = {PMLR},
  series    = {Proceedings of Machine Learning Research},
  volume    = {215},
}

@InProceedings{martin.belief2024,
  author    = {Martin, Ryan},
  booktitle = {BELIEF 2024},
  title     = {Which statistical hypotheses are afflicted by false confidence?},
  year      = {2024},
  address   = {Switzerland},
  editor    = {Bi, Y. and Jousselme, A.-L. and Denoeux, T.},
  pages     = {140--149},
  publisher = {Springer Nature},
  series    = {Lecture Notes in Artificial Intelligence},
  volume    = {14909},
}

@Article{martin.basu,
  author  = {Martin, Ryan},
  journal = {Sankhya A},
  title   = {A possibility-theoretic solution to {B}asu's {B}ayesian--frequentist via media},
  year    = {2024},
  pages   = {43--70},
  volume  = {86},
}

@Article{imreview,
  author  = {Martin, Ryan},
  journal = {J. Amer. Statist. Assoc.},
  title   = {Possibilistic inferential models: a review},
  year    = {2026},
  number  = {553},
  pages   = {807--826},
  volume  = {121},
}

@Unpublished{reimagined,
  author = {Martin, Ryan},
  note   = {{\em Statist. Sci.}, to appear, with discussion; {\tt arXiv:2503.19748}},
  title  = {No-prior {B}ayes re{IM}agined: probabilistic approximations of possibilistic inferential models},
  year   = {2026},
}

@Article{imbasics,
  author   = {Martin, Ryan and Liu, Chuanhai},
  title    = {Inferential models: a framework for prior-free posterior probabilistic inference},
  journal  = {J. Amer. Statist. Assoc.},
  year     = {2013},
  volume   = {108},
  number   = {501},
  pages    = {301--313},
  fjournal = {Journal of the American Statistical Association},
  mrclass  = {62A01 (62F15 62G10)},
  mrnumber = {3174621},
}

@Book{imbook,
  author    = {Martin, Ryan and Liu, Chuanhai},
  publisher = {CRC Press, Boca Raton, FL},
  title     = {Inferential {M}odels},
  year      = {2015},
  series    = {Monographs on Statistics and Applied Probability},
  volume    = {147},
  mrclass   = {62-02 (62A01 62F15)},
  mrnumber  = {3618727},
  pages     = {xix+254},
}

@Article{imdec.ext,
  author  = {Martin, Ryan and Prim, Shih-Ni and Williams, Jonathan},
  journal = {Internat. J. Approx. Reason.},
  title   = {Decision-making with possibilistic inferential models},
  year    = {2026},
  pages   = {Paper No. 109720},
  volume  = {196},
}

@Article{imbvm.ext,
  author   = {Martin, Ryan and Williams, Jonathan P.},
  journal  = {Internat. J. Approx. Reason.},
  title    = {Asymptotic efficiency of inferential models and a possibilistic {B}ernstein--von {M}ises theorem},
  year     = {2025},
  pages    = {Paper No. 109389},
  volume   = {180},
  fjournal = {International Journal of Approximate Reasoning},
  mrclass  = {68T37 (62F15)},
  mrnumber = {4869120},
}

@Book{mayo.book.1996,
  title     = {{Error and the Growth of Experimental Knowledge}},
  publisher = {University of Chicago Press},
  year      = {1996},
  author    = {Mayo, Deborah G.},
  address   = {Chicago},
}

@Book{mayo.book.2018,
  author    = {Mayo, Deborah G.},
  publisher = {Cambridge University Press},
  title     = {{S}tatistical {I}nference as {S}evere {T}esting},
  year      = {2018},
  address   = {Cambridge},
}

@Article{miranda2008,
  author   = {Miranda, Enrique},
  title    = {A survey of the theory of coherent lower previsions},
  journal  = {Internat. J. Approx. Reason.},
  year     = {2008},
  volume   = {48},
  number   = {2},
  pages    = {628--658},
  issn     = {0888-613X},
  fjournal = {International Journal of Approximate Reasoning},
  mrclass  = {68T37 (03B48 60A05)},
  mrnumber = {2418677},
}

@Article{neyman1955,
  author   = {Neyman, Jerzy},
  journal  = {Comm. Pure Appl. Math.},
  title    = {The problem of inductive inference},
  year     = {1955},
  issn     = {0010-3640},
  pages    = {13--45},
  volume   = {8},
  fjournal = {Communications on Pure and Applied Mathematics},
  mrclass  = {60.0X},
  mrnumber = {68145},
}

@Article{neymanpearson1933,
  author  = {Neyman, J. and Pearson, Egon S.},
  title   = {On the problem of most efficient tests of statistical hypotheses},
  journal = {Phil. Trans. Roy. Soc. London Ser. A},
  year    = {1933},
  volume  = {231},
  pages   = {289--337},
}

@Book{popper1959,
  title     = {The {L}ogic of {S}cientific {D}iscovery},
  publisher = {Hutchinson and Co., Ltd., London},
  year      = {1959},
  author    = {Popper, Karl R.},
  mrclass   = {02.00 (60.00)},
  mrnumber  = {0107593},
  pages     = {480},
}

@Book{popper1962,
  author    = {Popper, Karl R.},
  publisher = {Routledge \& Kegan Paul},
  title     = {Conjectures and {R}efutations: {T}he {G}rowth of {S}cientific {K}nowledge},
  year      = {1963},
  address   = {London},
}

@Book{popper1985.realism,
  author    = {Popper, Karl R.},
  publisher = {Routledge},
  title     = {Realism and the {A}im of {S}cience},
  year      = {1985},
  address   = {London},
  isbn      = {9780415084000},
}

@Article{popper.miller.1983,
  author  = {Popper, Karl R. and Miller, David W.},
  journal = {Nature},
  title   = {A proof of the impossibility of inductive probability},
  year    = {1983},
  number  = {21},
  pages   = {687--688},
  volume  = {302},
}

@Article{popper.miller.1987,
  author   = {Popper, K. R. and Miller, D. W.},
  journal  = {Philos. Trans. Roy. Soc. London Ser. A},
  title    = {Why probabilistic support is not inductive},
  year     = {1987},
  issn     = {0080-4614},
  number   = {1562},
  pages    = {569--591},
  volume   = {321},
  fjournal = {Philosophical Transactions of the Royal Society of London. Series A. Mathematical and Physical Sciences},
  mrclass  = {03B48 (03A05 60A05)},
  mrnumber = {892293},
}

@Article{redhead1985,
  author   = {Redhead, Michael},
  journal  = {British J. Philos. Sci.},
  title    = {On the impossibility of inductive probability},
  year     = {1985},
  issn     = {0007-0882},
  number   = {2},
  pages    = {185--191},
  volume   = {36},
  fjournal = {The British Journal for the Philosophy of Science},
  mrclass  = {03B48 (03A05)},
  mrnumber = {915926},
}

@Book{russell.history,
  author    = {Russell, Bertrand},
  publisher = {George Allen and Unwin Ltd.},
  title     = {A {H}istory of {W}estern {P}hilosophy},
  year      = {1946},
  address   = {London},
}

@Article{stigler.laplace,
  author  = {Stephen M. Stigler},
  journal = {Statist. Sci.},
  title   = {Laplace's 1774 memoir on inverse probability},
  year    = {1986},
  pages   = {359--363},
  volume  = {1},
}

@Book{lower.previsions.book,
  title     = {Lower {P}revisions},
  publisher = {John Wiley \& Sons, Ltd., Chichester},
  year      = {2014},
  author    = {Troffaes, Matthias C. M. and de Cooman, Gert},
  series    = {Wiley Series in Probability and Statistics},
  mrclass   = {60-02 (60Axx 60B99 62H20 62M40 68T37)},
  mrnumber  = {3222242},
  pages     = {xviii+415},
}

@Book{venn.chance,
  author    = {Venn, John},
  publisher = {MacMillan},
  title     = {The {L}ogic of {C}hance},
  year      = {1876},
  address   = {London},
}

@Book{walley1991,
  title     = {Statistical {R}easoning with {I}mprecise {P}robabilities},
  publisher = {Chapman \& Hall Ltd.},
  year      = {1991},
  author    = {Walley, Peter},
  volume    = {42},
  series    = {Monographs on Statistics and Applied Probability},
  address   = {London},
  mrclass   = {62A99 (60A05 62C99)},
  mrnumber  = {1145491},
  pages     = {xii+706},
}

@Article{zabell1989.rule,
  author  = {Zabell, Sandy},
  journal = {Erkenntnis},
  title   = {The rule of succession},
  year    = {1989},
  number  = {2--3},
  pages   = {283--321},
  volume  = {31},
}

@Article{zabell.attribution,
  author  = {Sandy L. Zabell},
  journal = {Archive for History of Exact Sciences},
  title   = {{Buffon, Price, and Laplace: Scientific attribution in the 18th century}},
  year    = {1988},
  number  = {2},
  pages   = {173--181},
  volume  = {39},
}

@Article{zadeh1975a,
  author   = {Zadeh, L. A.},
  journal  = {Information Sci.},
  title    = {The concept of a linguistic variable and its application to approximate reasoning. {I}},
  year     = {1975},
  pages    = {199--249},
  volume   = {8},
  mrclass  = {68A30 (02C05 68A45)},
  mrnumber = {0386369},
}

@Book{imprecise.prob.book,
  title     = {Introduction to {I}mprecise {P}robabilities},
  publisher = {John Wiley \& Sons, Ltd., Chichester},
  year      = {2014},
  editor    = {Augustin, Thomas and Coolen, Frank P. A. and de Cooman, Gert and Troffaes, Matthias C. M.},
  series    = {Wiley Series in Probability and Statistics},
  mrclass   = {60-06 (60-01 60-02 60A05 60A86 62-06 62C86 91-06)},
  mrnumber  = {3236913},
  pages     = {xxviii+404},
}

\end{document}